\documentclass{amsart}
\usepackage{amsmath}
\usepackage{amsfonts}
\usepackage{amssymb}
\usepackage{amsthm}

\newcommand{\ot}{\otimes}
\newcommand{\vt}{\overline{\otimes}}
\newcommand{\ft}{\overline{\otimes}_{\cl{F}}}
\newcommand{\cl}{\mathcal}
\newcommand{\sub}{\subseteq}
\newcommand{\bfc}{\textit}
\newcommand{\ap}{\infty}
\newcommand{\wsp}{\overline{\mathrm{span}}^{\text{w*}}}
\newcommand{\wh}{\widehat}
\newcommand{\wt}{\widetilde}
\newcommand{\BLT}{B(L^{2}(G))}
\newcommand{\LO}{L^{1}(G)}
\newcommand{\LT}{L^{2}(G)}
\newcommand{\LI}{L^{\infty}(G)}
\newcommand{\CI}{\mathbb{C}1_{H}}
\newcommand{\CIK}{\mathbb{C}1_{K}}
\newcommand{\la}{\langle}
\newcommand{\ra}{\rangle}
\newcommand{\fcr}{\rtimes^{\cl{F}}}
\newcommand{\scr}{\overline{\rtimes}}
\newcommand{\B}{\mathrm{Bim}(J^{\perp})}
\newcommand{\R}{(\mathrm{Ran}J)^{\perp}}

\newtheorem{thm}{Theorem}[section]
\newtheorem{lem}[thm]{Lemma}
\newtheorem{pro}[thm]{Proposition}
\newtheorem{cor}[thm]{Corollary}
\theoremstyle{definition}
\newtheorem{defin}[thm]{Definition}
\newtheorem{remark}[thm]{Remark}

\begin{document}
	 
\title[Crossed products of dual operator spaces]{Crossed products of dual operator spaces by locally compact groups}
\author{Dimitrios Andreou}
\address{Department of Mathematics, National and Kapodistrian University of Athens, Athens 157 84, Greece}
\email{dimandreou95@gmail.com}
\maketitle

\begin{abstract} 
	If $ \alpha $ is an action of a locally compact group $ G $ on a dual operator space $ X $, then two generally different notions of crossed products are defined, namely the Fubini crossed product $ X\fcr_{\alpha} G $ and the spatial crossed product $ X\scr_{\alpha} G $. It is shown that $ X\fcr_{\alpha}G=X\scr_{\alpha}G $ if and only if the dual comodule action $ \wh{\alpha} $ of the group von Neumann algebra $ L(G) $ on $ X\fcr_{\alpha}G $ is non-degenerate. As an application, this yields an alternative proof of the result of Crann and Neufang \cite{CN} that the two notions coincide when $ G $ satisfies the approximation property (AP) of Haagerup and Kraus. Also, it is proved that the $ L(G) $-bimodules $ \B $ and $ \R $ defined in \cite{AKT} for a left ideal $ J $ of $ \LO $ are respectively isomorphic with  $ J^{\perp}\scr G $ and $ J^{\perp}\fcr G $. Therefore a necessary and sufficient condition so that $ \B=\R $ is obtained by the main result.  
\end{abstract}

\section{Introduction}  

Let $ s\mapsto\alpha_{s},\ s\in G $ be an action of a locally compact group $ G $ on a von Neumann algebra $ M $ by unital normal *-automorphisms. Then, $ \alpha $ defines a unital normal *-monomorphism $ \alpha\colon M\to L^{\ap}(G,M)\cong M\vt\LI $ via $ \alpha(x)(s)=\alpha_{s}^{-1}(x) $, for $ s\in G $, $ x\in M $. The crossed product $ M\rtimes_{\alpha}G $ of $ M $ by $ \alpha $ is defined to be the von Neumann subalgebra of $ M\vt\BLT $ generated by $ \alpha(M) $ and $ \mathbb{C}1\vt L(G) $, where $ L(G)=\lambda(G)'' $ is the left group von Neumann algebra. There are two alternative ways to describe $ M\rtimes_{\alpha}G $:
\begin{itemize}
	\item[(1)] On the one hand, from the covariance relations \[\alpha(\alpha_{s}(x))=(1\ot\lambda_{s})\alpha(x)(1\ot\lambda_{s})^{*},\quad x\in M,\ s\in G, \]
	it follows that $ M\rtimes_{\alpha}G $ is the w*-closure of the linear span of the products of the form $ (1\ot\lambda_{s})\alpha(x) $, $ x\in M,\ s\in G $. That is, $ M\rtimes_{\alpha}G $ is the normal $ L(G) $-bimodule generated by $ \alpha(M) $.
	\item[(2)] On the other, from the Digernes-Takesaki theorem (see e.g. \cite[Chapter X, \S1, Corollary 1.22]{Tak}) we have
	\[M\rtimes_{\alpha}G=\{x\in M\vt\BLT:\ (\alpha_{s}\ot \mathrm{Ad}\rho_{s})(x)=x,\ \forall s\in G \}, \]
	where $ \rho $ is the right regular representation of $ G $.  
\end{itemize}

These two alternative descriptions of the crossed product $ M\rtimes_{\alpha}G $ can both be generalized for group actions on dual operator spaces. Therefore, for a group action $ \alpha $ on a dual operator space $ X\sub B(H) $ by w*-continuous completely isometric isomorphisms, one can define again a w*-continuous complete isometry $ \alpha\colon X\to  X\vt\LI  $ and obtain the spatial crossed product
\[X\scr_{\alpha}G=\wsp\{(1_{H}\ot\lambda_{s})\alpha(x):\ s\in G,\ x\in X \} \]
and the Fubini crossed product
\[X\fcr_{\alpha}G=\{x\in X\vt\BLT:\ (\alpha_{s}\ot \mathrm{Ad }\rho_{s})(x)=x,\ \forall s\in G \}. \] 
 
The equality $ X\fcr_{\alpha}G=X\scr_{\alpha}G $ does not hold in general (see Remark \ref{rem7}). However, P. Salmi and A. Skalski in \cite{SS} have proved a generalization of the Digernes-Takesaki theorem in the case of a group action on a W*-TRO by TRO-morphisms. More precisely they proved that if $ X $ is a W*-TRO, $ G $ is an arbitrary locally compact group and $ \alpha $ is a W*-TRO morphism such that $ \alpha\colon X\to X\vt\LI $ is a non-degenerate TRO-morphism, then $ X\fcr_{\alpha}G= X\scr_{\alpha}G $ (for the definition of a non-degenerate TRO-morphism see e.g. Remark \ref{rem6}).

On the other hand, J. Crann and M. Neufang in \cite{CN} proved  that if  $ G $ is a locally compact group with the approximation property of U. Haagerup and J. Kraus \cite{HK} (for more details see section \ref{sec4} below) and $ X $ is an arbitrary dual operator space, then $ X\fcr_{\alpha}G=X\scr_{\alpha}G $. Also, they obtained the analogous result for group actions on general operator spaces (not necessarily dual), generalizing a result of O. Uuye and J. Zacharias for discrete group actions on operator spaces (see \cite{UZ}). 

The analogue of the Fubini crossed product, in the setting of general operator spaces, had already been studied by M. Hamana in \cite{Ha1} who proved a Takesaki-type duality theorem for Fubini crossed products.

The main purpose of this paper is to characterize those group actions on dual operator spaces for which the Fubini crossed product coincides with the spatial crossed product. More precisely, it is shown that, for an action $ \alpha $ of a locally compact group $ G $ on a dual operator space $ X $, we have $ X\fcr_{\alpha}G=X\scr_{\alpha}G $ if and only if the dual comodule action $ \wh{\alpha} $ of the group von Neumann algebra $ L(G) $ on $  X\fcr_{\alpha}G $ is non-degenerate (see Definitions \ref{def4.2}  and \ref{def2.16}). As applications this yields alternative proofs of the results of Crann and Neufang (see Proposition \ref{pro7.2}) and Salmi and Skalski (see Remark \ref{rem8}) mentioned above. We also prove that the $ L(G) $-bimodules $ \B $ and $ \R $ defined in \cite{AKT} for a closed left ideal $ J $ of $ \LO $ can both be realized as crossed products. This fact provides a perhaps more conceptual perspective for these bimodules; it also  yields a necessary and sufficient condition for their equality, as an application of our theorem.

In section 2 below we begin with preliminaries and notation on tensor products, comodules over Hopf-von Neumann algebras and crossed products of von Neumann algebras. In section 3 we compare the two notions of crossed products under discussion and we prove our main result (see Theorem \ref{mainthm}). In the same section we obtain an alternative proof of the result of Salmi and Skalski. In section 4 we give an altrernative proof of the result of  Crann and Neufang. In section 5 it is shown that  $ \B $ and $ \R $ arise naturally as crossed products and a necessary and sufficient condition for their equality is obtained. 

\section{Preliminaries and notation}    

Let $ X\sub B(H) $ and $ Y\sub B(K) $ be dual operator spaces, i.e. w*-closed subspaces of $ B(H) $ and $ B(K) $ respectively, where $ H,\ K $ are Hilbert spaces. The \bfc{spatial tensor product} of $ X $ and $ Y $ is the subspace of $ B(H\ot K) $ defined by \[X\vt Y=\wsp\{x\ot y:\ x\in X,\ y\in Y \}, \]
	where $ (x\ot y)(h\ot k) =(xh)\ot(yk)$, for $ h\in H ,\ k\in K$.
	
	 The \bfc{Fubini tensor product} of $ X $ and $ Y $ is the space:
	 \begin{align*} 
	    X\ft Y=\{x\in B(H\ot K) :\  
	    &(\mathrm{id}_{B(H)}\ot \phi)(x)\in X,\ (\omega\ot \mathrm{id }_{B(K)})(x)\in Y,\\  &\forall\omega\in B(H)_{*},\ \forall\phi\in B(K)_{*} \}.
	 \end{align*}
	
	Obviously, $ X\vt Y\sub X\ft Y $. 
	
	 If $ M $ is an injective von Neumann algebra (in particular, 
	of type I) then $ X\vt M= X\ft M $, for every dual operator space $ X $ (see \cite[Theorem 1.9]{Kra}) and this implies that	
	  \[X\ft Y=(X\vt B(K)) \cap (B(H)\vt Y).\]   
	 Also, for any von Neumann algebras $ M$ and $ N $, it holds that $ M\vt N= M\ft N $ (see \cite[Theorem 7.2.4]{ER}).	  

	A \bfc{Hopf-von Neumann algebra} is a pair $ (M,\Delta) $, where $ M $ is a von Neumann algebra and $ \Delta\colon M\to M\vt M $ is a normal unital *-monomorphism called the \bfc{comultiplication} of $ M $, such that the coassociativity rule holds:
	\[(\Delta\otimes \mathrm{id }_{M})\circ\Delta=(\mathrm{id }_{M}\otimes\Delta)\circ\Delta. \]	
	 
 Let $ (M,\Delta) $ be a Hopf von Neumann algebra. An	$ M $-\bfc{comodule} $ (X,\alpha) $ is a dual operator space $ X $ with a w*-continuous complete isometry $ \alpha\colon X\to X\ft M $ which satisfies
	\[(\alpha\otimes \mathrm{id }_{M})\circ\alpha=(\mathrm{id }_{X}\otimes\Delta)\circ\alpha. \]
	In this case, we say that $ \alpha $ is an \bfc{action} of $ M $ on $ X $ or an $ M $-\bfc{action} on $ X $.
		
	A w*-closed subspace $ Y $ of $ X $ is called an $ M $-\bfc{subcomodule} of $ X $ if $\alpha(Y)\sub Y\ft M$. In this case we write $ Y\leq X $ and $ Y $ is indeed an $ M $-comodule for the action $ \alpha|_{Y} $. 
	
	An $ M $-\bfc{comodule morphism} between $ M $-comodules $ (X,\alpha) $ and $ (Y,\beta) $ is a w*-w*-continuous complete contraction $ \phi\colon X\to Y $, such that
	\[\beta\circ\phi=(\phi\otimes \mathrm{id }_{M})\circ\alpha. \]
	
	An $ M $-comodule morphism is called an $ M $-\bfc{comodule monomorphism} (resp. isomorphism) if it is a complete isometry (resp. surjective complete isometry) and we write $ X\cong Y $ for isomorphic $ M $-comodules.
	
	If $ X $ is any dual operator space, then the Fubini tensor product $ X\ft M $ becomes an $ M $-comodule with the action $ \mathrm{id }_{X}\otimes\Delta\colon X\ft M\to X\ft M\ft M$, which is called a \bfc{canonical} $ M $-comodule.
	
	If $ N $ is a von Neumann algebra, then  a normal unital *-monomorphism $ \pi\colon N\to N\vt M $, such that 
	\[(\pi\otimes \mathrm{id }_{M})\circ\pi=(\mathrm{id }_{N}\otimes\Delta)\circ\pi\]
	will be called a W*-$ M $-\bfc{action} of $ M $ on $ N $ and $ (N,\pi) $ will be called a W*-$ M $-\bfc{comodule}. The terms W*-$ M $-\bfc{subcomodule}, W*-$ M $-\bfc{comodule morphism} etc, are defined accordingly. 

\begin{remark}\label{rem1} 
	 Let $ (M,\Delta) $ be a Hopf-von Neumann algebra. Every $ M $-comodule $ (X,\alpha) $ is isomorphic to an $ M $-subcomodule of a canonical $ M $-comodule, which may be taken to be the W*-$ M $-comodule $ (B(H)\vt M,\mathrm{id }_{B(H)}\otimes\Delta) $ for some Hilbert space $ H $.\\
	 Indeed, the image $ \alpha(X) $ under the action $ \alpha $ is an $ M $-subcomodule of the canonical $ M $-comodule $ X\ft M $, since we have:
	 \[(\mathrm{id }_{X}\otimes\Delta)\circ\alpha(X)=(\alpha\otimes \mathrm{id }_{M})\circ\alpha(X)\sub\alpha(X)\ft M \] 
	 and $ \alpha $ is an $ M $-comodule isomorphism of $ X $ onto $ \alpha(X) $ and thus
	 \[X\cong\alpha(X)\leq X\ft M. \]
	 Furthermore, we may suppose that $ X\sub B(H) $ as a w*-closed subspace for some Hilbert space $ H $, thus $ X\ft M \leq B(H)\vt M $.
\end{remark}

\begin{remark}\label{rem3}
	For any Hopf von Neumann algebra $ (M,\Delta) $, the predual $ M_{*} $ of $ M $ becomes naturally a Banach algebra with the product defined as:
	\[f\cdot g=(f\otimes g)\circ\Delta, \]
	for $ f,\ g\in M_{*} $. Furthermore, an $ M $-comodule $ (X,\alpha) $ becomes an $ M_{*} $-Banach module with the module operation defined as
	\[f\cdot x=(\mathrm{id }_{X}\otimes f)\circ\alpha(x),\quad f\in  M_{*},\ x\in X. \]
	Also, it is easy to see that a w*-continuous complete contraction $ \phi\colon X\to Y $ between two $ M $-comodules $ X $ and $ Y $ is an $ M $-comodule morphism if and only if $ \phi $ is an $ M_{*} $-module homomorphism.
	\end{remark}

	Let $ (X,\alpha) $ be an $ M $-comodule over a Hopf-von Neumann algebra $ (M,\Delta) $. The \bfc{fixed point subspace} of $ X $ is the operator space
	\[X^{\alpha}=\{x\in X:\ \alpha(x)=x\otimes1_{M} \}.\] 	

Note that $ X^{\alpha} $ is obviously an $ M $-subcomodule of $X$.

Another important notion concerning actions of Hopf-von Neumann algebras is \bfc{commutativity of actions}:

\begin{defin}\label{def2.3} 
	Let $ (M_{1},\Delta_{1}) $ and $ (M_{2},\Delta_{2}) $ be two Hopf-von Neumann algebras and $ \alpha_{1} $,  $ \alpha_{2} $ be actions of $ M_{1} $ and $ M_{2} $ on the same operator space $ X $ respectively. We say that $ \alpha_{1} $ and  $ \alpha_{2} $ \bfc{commute} if
	\[(\alpha_{1}\otimes \mathrm{id }_{M_{2}})\circ\alpha_{2}=(\mathrm{id }_{X}\otimes\sigma)\circ(\alpha_{2}\otimes \mathrm{id }_{M_{1}})\circ\alpha_{1}, \]
	where $ \sigma\colon M_{2}\vt M_{1}\to M_{1}\vt M_{2}:\ x\otimes y\mapsto y\otimes x $ is the flip isomorphism.
\end{defin}
\begin{lem}\label{lem2.3}\cite[Lemma 5.2]{Ha1} 
	With the notation and assumptions of Definition \ref{def2.3}, we have that the fixed point subspace $X^{\alpha_{1}} $ is an $ M_{2} $-subcomodule of $ (X,\alpha_{2}) $, i.e. the restriction $ \alpha_{2}|_{X^{\alpha_{1}}} $ is an action of $ M_{2} $ on $ X^{\alpha_{1}} $.
\end{lem}

For the rest of this paper, $ G $ will denote a locally compact (Hausdorff) group with left Haar measure $ ds $ and modular function $ \Delta_{G} $. We identify $L^{\infty}(G)  $ with the multiplicative operators acting on $ L^{2}(G) $. Also, the fundamental unitary operator $ V_{G}\in B(L^{2}(G))\vt B(L^{2}(G))\simeq B(L^{2}(G\times G)) $ defined by the formula 
\[V_{G}f(s,t)=f(t^{-1}s,t),\quad f\in L^{2}(G\times G),\ s,\ t\in G, \]
gives rise to a comultiplication $ \alpha_{G}\colon L^{\infty}(G)\to L^{\infty}(G)\vt L^{\infty}(G) $ on $ L^{\infty}(G) $ via 
\[\alpha_{G}(f)=V_{G}^{*}(f\otimes1)V_{G} \] 
and it is easy to see that 
\[ \alpha_{G}(f)(s,t)=f(ts),\quad s,\ t\in G, \] (identifying $ L^{\infty}(G\times G)\simeq L^{\infty}(G)\vt L^{\infty}(G) $). Thus, $ (L^{\infty}(G),\alpha_{G}) $ is a Hopf von Neumann algebra.

On the other hand, the unitary $ U_{G}\in\BLT\vt\BLT $, defined by
\[U_{G}\xi(s,t)=\Delta_{G}(t)^{1/2}\xi(st,t),\quad s,\ t\in G,\ \xi\in L^{2}(G\times G), \]
induces another comultiplication  on $ \LI $, namely  
\[ \alpha'_{G}(f)=U_{G}(f\otimes1)U_{G}^{*},\quad f\in\LI, \]
or equivalently 
\[\alpha'_{G}(f)(s,t)=f(st),\quad s,\ t\in G. \]
Obviously, we have
\[\alpha_{G}=\sigma\circ\alpha'_{G}, \]
where $ \sigma\colon\BLT\vt\BLT\to\BLT\vt\BLT:\ x\ot y\mapsto y\ot x $, is the flip mapping.

Another basic example of a Hopf-von Neumann algebra, is the left von Neumann algebra of $ G $, i.e. the algebra $ L(G):= \lambda(G)''\subseteq B(L^{2}(G)) $ generated by the left regular representation $ \lambda\colon G\ni s\mapsto\lambda_{s}\in B(L^{2}(G)) $,
\[\lambda_{s}\xi(t)=\xi(s^{-1}t) , \qquad\xi\in L^{2}(G). \]
The unitary operator $ W_{G}\in B(L^{2}(G\times G)) $, given by the formula
\[W_{G}f(s,t)=f(s,st),\quad f\in L^{2}(G\times G),\ s,\ t\in G, \]
induces the comultiplication $ \delta_{G}\colon L(G)\to L(G)\vt L(G) $ on $ L(G) $ via
\[\delta_{G}(x)=W_{G}^{*}(x\otimes1)W_{G}. \]
It is easy to verify that 
\[\delta_{G}(\lambda_{s})=\lambda_{s}\otimes\lambda_{s},\quad s\in G. \]

Recall the Fourier algebra $ A(G) $ of $ G $ (see e.g. \cite{Ey}):
\[A(G)=\{u\colon G\to \mathbb{C}:\ \exists\xi, \eta\in\LT,\ \forall s\in G,\ u(s)=\la \lambda_{s}\xi,\ \eta \ra \}. \]
With the pointwise product, $ A(G) $  is a Banach algebra with respect to the norm
\[||u||_{A(G)}=\inf\{||\xi||||\eta||:\ \xi, \eta\in\LT, \text{ such that } u(s)=\la \lambda_{s}\xi,\ \eta \ra  \} \]
and it is isometrically isomorphic with the predual $ L(G)_{*} $ of the group von Neumann algebra, the duality given by
\[\la\lambda_{s},\ u\ra=u(s),\qquad s\in G,\ u\in A(G). \]

The pointwise product on $ A(G) $ coincides with that induced on the predual $ L(G)_{*} $ by the comultiplication $ \delta_{G} $ of $ L(G) $, because:
\[\la \lambda_{s} ,\ uv\ra=u(s)v(s)=\la \lambda_{s} ,\ u\ra\la \lambda_{s} ,\ v\ra=\la \lambda_{s}\ot\lambda_{s} ,\ u\ot v\ra=\la \delta_{G}(\lambda_{s}),\ u\ot v\ra. \]
In the following, $ A(G) $ will be deliberately identified with $ L(G)_{*} $.

Note that $ V_{G}\in L(G)\vt L^{\infty}(G) $ and $ W_{G}\in L^{\infty}(G)\vt L(G) $. Furthermore, $ \alpha_{G} $ and $ \delta_{G} $ extend to actions of $ L^{\infty}(G) $ and $ L(G) $ on $ B(L^{2}(G)) $ respectively, via the formulas
\[ \alpha_{G}(x)=V_{G}^{*}(x\otimes1)V_{G},\quad x\in B(L^{2}(G)),  \]
\[\delta_{G}(x)=W_{G}^{*}(x\otimes1)W_{G}, \quad x\in B(L^{2}(G)). \]

Also, recall the right  regular representation of $ G $ on $ L^{2}(G) $:
\[\rho_{s}f(t)=\Delta_{G}(s)^{1/2}f(ts),\qquad s,\ t\in G,\ f\in L^{2}(G). \]
We denote by $ R(G)=\rho(G)'' $ the right group von Neumann algebra and it is easy to verify as well that $ U_{G}\in R(G)\vt L^{\ap}(G) $.

In the following, we assume always that $ \LI $ and $ L(G) $ are considered as Hopf-von Neumann algebras with respect to $ \alpha_{G} $ and $ \delta_{G} $ respectively.

Now, let us collect some well known facts about crossed products of von Neumann algebras (see for example \cite{NaTa}, \cite{Stra}, \cite{SVZ1}, \cite{SVZ2}) and \cite{Tak}), which will be used in the sequel:

\begin{thm}\label{thm1.1}
	Given a (pointwise) group action on a von Neumann algebra $ M $, i.e. a w*-continuous representation $\alpha\colon G\to \mathrm{Aut}(M) $ by *-automorphisms of the von Neumann algebra $ M $, for any $ x\in M $ the w*-continuous function $ s\mapsto\alpha^{-1}_{s}(x) $ defines an element $ \pi_{\alpha}(x)\in L^{\infty}(G,M)\simeq M\vt L^{\infty}(G) $ and the map $ \pi_{\alpha}\colon M\to M\vt L^{\infty}(G) $ is an action of $ (L^{\infty}(G),\alpha_{G}) $ on $ M $, i.e. a unital normal *-monomorphism satisfying
	\[(\pi_{\alpha}\otimes \mathrm{id }_{L^{\infty}(G)})\circ\pi_{\alpha}=(\mathrm{id }_{M}\otimes\alpha_{G})\circ\pi_{\alpha} \]
	Conversely, for every W*-$ \LI $-action $ \pi\colon M\to M\vt L^{\infty}(G)  $ of $ (L^{\infty}(G),\alpha_{G}) $ on $ M $, there exists a unique w*-continuous representation $\alpha\colon G\to \mathrm{Aut}(M)  $ by *-automorphisms of $ M $, such that $ \pi_{\alpha}=\pi $. 
\end{thm}

The \bfc{crossed product} of $ M $ by the action $ \alpha $ is defined to be the von Neumann subalgebra of $ M\vt B(L^{2}(G)) $ generated by $ \pi_{\alpha}(M) $ and $ \mathbb{C}\vt L(G) $, that is 
\[M\rtimes_{\alpha}G=(\pi_{\alpha}(M)\cup\mathbb{C}\vt L(G))''. \]
Because of the covariance relations:\[\pi_{\alpha}(\alpha_{s}(x))=(1\ot\lambda_{s})\pi_{\alpha}(x)(1\ot\lambda_{s})^{*},\quad x\in M,\ s\in G, \]
we get that \[M\rtimes_{\alpha}G =\wsp\{(1\ot\lambda_{s})\pi_{\alpha}(x):\ s\in G,\ x\in M \},\]
where span denotes the linear span.

For a group action $ \alpha $ on $ M $ we set $M^{\alpha}:=\{x\in M:\ \alpha_{s}(x)=x,\ \forall s\in G \} $.
Then, we have that $ M^{\pi_{\alpha}}=M^{\alpha}$, i.e.
\[\pi_{\alpha}(x)=x\otimes1\iff\alpha_{s}(x)=x,\ \forall s\in G. \]

Fix a group action  $ \alpha\colon G\to \mathrm{Aut}(M) $ on a von Neumann algebra $ M $ and consider the action $ \beta=\alpha\otimes \mathrm{Ad }\rho $ on $ M\vt B(L^{2}(G)) $, that is $ \beta_{s}=\alpha_{s}\otimes \mathrm{Ad }\rho_{s} $, for all $ s\in G $.

\begin{pro}[Dual action]
	If we set 
	\[ \widehat{\alpha}=(\mathrm{id }_{M}\otimes\delta_{G})|_{ M\rtimes_{\alpha}G}, \]
	where $ \delta_{G} $ is considered as an $ L(G) $-action on $ \BLT $ and $ \mathrm{id }_{M}\otimes\delta_{G} $ is an $ L(G) $-action on $ M\vt\BLT $, then $ \widehat{\alpha} $ is a W*-$ L(G) $-action on $ M\rtimes_{\alpha}G $, which is called the \bfc{dual} of $ \alpha $. Note that 
	\[\widehat{\alpha}(x)=\mathrm{Ad }_{(1\otimes W_{G}^{*})}(x\otimes1),\quad x\in M\rtimes_{\alpha}G. \]
	
\end{pro}

\begin{pro}\label{pro4.1} 
	 The corresponding W*-$ \LI $-actions $ \pi_{\alpha} $ and $ \pi_{\beta} $ satisfy:
	\[\pi_{\beta}=(\mathrm{id }_{M}\otimes \mathrm{Ad }U_{G}^{*})\circ(\mathrm{id }_{M}\otimes\sigma)\circ(\pi_{\alpha}\otimes \mathrm{id }_{B(L^{2}(G))})\ , \]
	where $ U_{G} $ is the unitary defined above and $ \sigma $ is the flip mapping on $\BLT\vt$ $\BLT$. Furthermore, $ \pi_{\beta} $ commutes with the $ L(G) $-action $ \mathrm{id }_{M}\otimes\delta_{G} $ (see Definition \ref{def2.3}).
\end{pro}

\begin{lem}\label{lem3.2} 
	The action $ \beta $ of $ G $ on $ M\vt B(L^{2}(G)) $ satisfies:
	\[(M\vt L^{\ap}(G))^{\pi_{\beta}}=(M\vt L^{\ap}(G))^{\beta}=\pi_{\alpha}(M). \]
\end{lem}

\begin{pro}\label{pro3.2} 
	For the dual action $ \wh{\alpha} $ of $ \alpha $ we have:
	 \[ (M\rtimes_{\alpha}G)^{\widehat{\alpha}}=\pi_{\alpha}(M). \]		
\end{pro}

\begin{thm}[Digernes-Takesaki]\label{thm3.2} 
	The crossed product $ M\rtimes_{\alpha}G $ is the fixed point space of $ \beta $, i.e.
	\[M\rtimes_{\alpha}G=(M\vt B(L^{2}(G)))^{\beta}=(M\vt B(L^{2}(G)))^{\pi_{\beta}}. \]	
\end{thm}

\section{Crossed products of dual operator spaces and main results}

  In this section, we prove some basic facts about crossed products of dual operator spaces which are generalizations of known results for crossed products of von Neumann algebras and we characterize those $ \LI $-comodules for which the Fubini crossed product coincides with the spatial crossed product (see Theorem \ref{mainthm}).

\begin{defin}[M. Hamana, \cite{Ha1}] 
	For an $ \LI $-comodule $ (X,\alpha) $, we define the map \[ \widetilde{\alpha} \colon X \vt B(L^{2}(G)) \to X \vt B(L^{2}(G)) \vt L^{\ap}(G) \] by
	\[\widetilde{\alpha}=(\mathrm{id }_{X}\otimes \mathrm{Ad }U_{G}^{*})\circ(\mathrm{id }_{X}\otimes\sigma)\circ(\alpha\otimes \mathrm{id }_{B(L^{2}(G))}) \]
	where $ \sigma $ is the flip mapping on $\BLT\vt\BLT$.
\end{defin}

The proof of the next result is essentially the same as the proof of Lemma 5.3 (i) in \cite{Ha1} and so we omit it. 
\begin{pro}\label{pro4.2} 
	Let $ (X,\alpha) $ be an $ \LI $-comodule. Then, $ \widetilde{\alpha} $ is an $ \LI $-action on $ X \vt B(L^{2}(G)) $, which commutes with the $ L(G) $-action $ \mathrm{id }_{X}\otimes\delta_{G} $.
\end{pro}

\begin{defin}[M. Hamana, \cite{Ha1}]\label{def4.2}
	Let $ (X,\alpha) $ be an $ \LI $-comodule. The \bfc{Fubini crossed product} of $ X $ by $ \alpha $ is defined to be the $ L(G) $-comodule $ (X\rtimes^{\cl{F}}_{\alpha}G,\ \widehat{\alpha}) $, where
	\[X\rtimes^{\cl{F}}_{\alpha}G:=(X\vt B(L^{2}(G)))^{\widetilde{\alpha}} \]
	and \[\widehat{\alpha}:=(\mathrm{id }_{X}\otimes\delta_{G})|_{X\rtimes_{\alpha}G}. \]
	The $ L(G) $-action $ \widehat{\alpha}\colon X\fcr_{\alpha}G\to (X\fcr_{\alpha}G)\ft L(G) $ is called the \bfc{dual} of $ \alpha $. 
\end{defin} 

\begin{remark}\label{rem4}
It follows immediately from Proposition \ref{pro4.2} and Lemma \ref{lem2.3}, that the dual action $\widehat{\alpha}$ is indeed an $ L(G) $-action on the Fubini crossed product $ X\rtimes^{\cl{F}}_{\alpha}G $, since $ \mathrm{id }_{X}\otimes\delta_{G} $ commutes with $ \wt{\alpha} $.
\end{remark}

\begin{defin}\label{def2.11} Let $ (X,\alpha) $ be an $ \LI $-comodule and suppose that $ X $ is w*-closed in $ B(H) $, for some Hilbert space $ H $. The \bfc{spatial crossed product} of $ X $ by $ \alpha $ is defined to be the space
	\[X\scr_{\alpha}G:=\wsp\{(1_{H}\ot\lambda_{s})\alpha(x):\ s\in G,\ x\in X \}\sub B(H)\vt\BLT. \] 	
\end{defin}

\begin{remark}\label{rem7}
	It is clear, from the Digernes-Takesaki theorem (see Theorem \ref{thm3.2} above), Theorem \ref{thm1.1} and Proposition \ref{pro4.1}, that $ X\fcr_{\alpha}G=X\scr_{\alpha}G $ when $ X $ is a von Neumann algebra and $ \alpha $ is in addition a unital normal *-homomorphism.
	
	However, this is not true for general $ \LI $-comodules. For example, take any discrete group failing the approximation property (see section \ref{sec4}), e.g. $ G=SL(3,\mathbb{Z}) $. Then, $ L(G) $ does not have the dual slice map property (see \cite[Theorem 2.1]{HK}), which means that there exists a dual operator space $ X $, such that $ X\vt L(G)\subsetneqq X\ft L(G) $.
	Consider the trivial $ \LI $-action $ \alpha\colon X\to X\vt \LI,\ \alpha(x)=x\ot1 $, for any $ x\in X $. Then, obviously, we have
	\[X\scr_{\alpha}G=X\vt L(G). \]
	On the other hand, it is not hard to see that, if $ \alpha $ is trivial, then \[X\fcr_{\alpha}G= X\ft L(G) \]
	and therefore $ X\fcr_{\alpha}G\neq X\scr_{\alpha}G $. 
\end{remark}

\begin{remark}\label{rem5} Let $ H,\ K $ be Hilbert spaces, $ X\sub B(H) $ a w*-closed subspace and $ b,\ c\in B(K) $. Then, we have
	\[(1_{H}\ot b)(X\vt B(K))(1_{H}\otimes c)\sub X\vt B(K). \]	
Indeed, for all $ a,\ b,\ c\in B(K) $ and $ v\in X $, we have $ (1_{H}\ot b)(v\ot a)(1_{H}\otimes c)=v\ot (bac)\in X\vt B(K) $, thus the above inclusion follows from the definition of the spatial tensor product and the fact that the multiplication in $ B(H)\vt B(K) $ is separately w*-continuous. As a consequence, if $ (X,\alpha) $ is an $ \LI $-comodule, then $ X\scr_{\alpha}G\sub X\vt \BLT $, because $ \alpha(X)\sub X\vt \LI\sub X\vt\BLT $. 

Also, if in addition $ Y $ is a w*-closed subspace of $ B(L) $ for some Hilbert space $ L $ and $ \phi\colon X\to Y $ is a w*-continuous completely bounded map, then $ \phi\ot \mathrm{id }_{B(K)}  \colon$ $ X \vt B(K) \to Y \vt B(K) $ is a w*-continuous $ B(K) $-bimodule map in the sense that
\[(\phi\ot \mathrm{id }_{B(K)})((1_{H}\ot a)x(1_{H}\ot b))=(1_{L}\ot a)(\phi\ot \mathrm{id }_{B(K)})(x)(1_{L}\ot b), \]
for all $ a,\ b\in B(K) $ and $ x\in X\vt B(K) $.\\
\end{remark}

\begin{pro}\label{pro2.13} Let $ (X,\alpha) $ be an $ \LI $-comodule and suppose that $ X $ is w*-closed in $ B(H) $, for some Hilbert space $ H $. Then, $ X\fcr_{\alpha}G $ is an $ L(G) $-bimodule, i.e.
	\[(1_{H}\ot\lambda_{s})y(1_{H}\ot\lambda_{t})\in X\fcr_{\alpha}G,\qquad s,t\in G,\ y\in X\fcr_{\alpha}G \]
	and $ \alpha(X)\sub X\fcr_{\alpha}G $. Therefore, we have: 
	\[X\scr_{\alpha}G\sub X\fcr_{\alpha}G. \]
	Furthermore, $ \wh{\alpha}(X\scr_{\alpha}G)\sub (X\scr_{\alpha}G)\ft L(G) $, that is $ X\scr_{\alpha}G $ is an $ L(G) $-subcomodule of $ (X\fcr_{\alpha}G,\wh{\alpha}) $.	
\end{pro}
\begin{proof} Let $ s\in G $ and $ y\in  X\fcr_{\alpha}G  $. Then, by Remark \ref{rem5} we have that $ (1_{H}\ot\lambda_{s})y\in X\vt \BLT $ and $ \wt{\alpha}(y)=y\ot 1 $, by Definition \ref{def4.2}. Also, by Remark \ref{rem5}, we have that
	\[(\alpha\ot \mathrm{id }_{\BLT})((1_{H}\ot\lambda_{s})y)=(1_{H}\ot1_{\LT}\ot\lambda_{s})(\alpha\ot \mathrm{id }_{\BLT})(y). \]
Thus, we have:
	\begin{align*} 
	   \wt{\alpha}((1_{H}\ot\lambda_{s})y)&=(\mathrm{id }_{X}\otimes \mathrm{Ad }U_{G}^{*})\circ(\mathrm{id }_{X}\otimes\sigma)\circ(\alpha\otimes \mathrm{id }_{B(L^{2}(G))})((1_{H}\ot\lambda_{s})y)\\
	   &=(\mathrm{id }_{X}\otimes \mathrm{Ad }U_{G}^{*})\circ(\mathrm{id }_{X}\otimes\sigma)\left((1_{H}\ot1_{\LT}\ot\lambda_{s})(\alpha\ot \mathrm{id }_{\BLT})(y) \right)\\
	   &=\left[ (\mathrm{id }_{B(H)}\otimes \mathrm{Ad }U_{G}^{*})\circ(\mathrm{id }_{B(H)}\otimes\sigma)((1_{H}\ot1_{\LT}\ot\lambda_{s}))\right] \wt{\alpha}(y)\\
	   &=\left[(1_{H}\ot U_{G}^{*})(1_{H}\ot\lambda_{s}\ot 1_{\LT})(1_{H}\ot U_{G}) \right](y\ot1_{\LT})\\
	   &=(1_{H}\ot\lambda_{s})y\ot1_{\LT}, 
	\end{align*}
where the third equality above follows from the fact that $ (\mathrm{id }_{B(H)}\otimes \mathrm{Ad }U_{G}^{*})\circ(\mathrm{id }_{B(H)}\otimes\sigma) $ is a *-homomorphism and thus multiplicative, while the last equality is because $ U_{G}\in R(G)\vt\LI $ and $ R(G)=L(G)' $. Therefore, $ (1_{H}\ot\lambda_{s})y\in X\fcr_{\alpha}G$. Similarly, we get $ y(1_{H}\ot\lambda_{t})\in X\fcr_{\alpha}G $ for all $ t\in G  $ and $ y\in X\fcr_{\alpha}G $.\\
On the other hand, if $ x\in X $, then:
    \begin{align*} 
       \wt{\alpha}(\alpha(x))&=(\mathrm{id }_{X}\otimes \mathrm{Ad }U_{G}^{*})\circ(\mathrm{id }_{X}\otimes\sigma)\circ(\alpha\otimes \mathrm{id }_{B(L^{2}(G))})(\alpha(x))\\
       &=(\mathrm{id }_{X}\otimes \mathrm{Ad }U_{G}^{*})\circ(\mathrm{id }_{X}\otimes\sigma)\circ(\mathrm{id }_{X}\ot\alpha_{G})(\alpha(x))\\
       &=(\mathrm{id }_{X}\otimes \mathrm{Ad }U_{G}^{*})\circ(\mathrm{id }_{X}\ot\alpha'_{G})(\alpha(x))\\
       &=(1_{H}\ot U_{G}^{*})(1_{H}\ot U_{G})(\alpha(x)\ot 1_{\LT})(1_{H}\ot U_{G}^{*})(1_{H}\ot U_{G})\\
       &=\alpha(x)\ot 1_{\LT},
    \end{align*}
because $ \alpha'_{G}=\sigma\circ\alpha_{G} $ and $ \alpha'_{G}(f)=U_{G}(f\ot1)U_{G}^{*} $, for all $ f\in \LI $. Hence, $ \alpha(X)\sub X\fcr_{\alpha}G $.\\
Finally, for $ x\in X $ and $ s\in G $, we have:
    \begin{align*} 
        \wh{\alpha}((1_{H}\ot\lambda_{s})\alpha(x))&=(\mathrm{id }_{B(H)}\ot\delta_{G})((1_{H}\ot\lambda_{s})\alpha(x))\\
        &=(\mathrm{id }_{B(H)}\ot\delta_{G})(1_{H}\ot\lambda_{s}))(\mathrm{id }_{B(H)}\ot\delta_{G})(\alpha(x))\\
        &=(1_{H}\ot\delta_{G}(\lambda_{s}))(1_{H}\ot W_{G}^{*})(\alpha(x)\ot1_{\LT})(1_{H}\ot W_{G})\\
        &=(1_{H}\ot\lambda_{s}\ot\lambda_{s})(\alpha(x)\ot1_{\LT}),  
    \end{align*}
because $ 1_{H}\ot W_{G}\in \CI\vt\LI\vt L(G) $ commutes with $ \alpha(x)\ot1_{\LT}\in B(H)\vt$ $\LI\vt\mathbb{C}1_{\LT} $. Therefore, we get:
\[ \wh{\alpha}((1_{H}\ot\lambda_{s})\alpha(x))=((1_{H}\ot\lambda_{s})\alpha(x))\ot\lambda_{s} \]
and it follows that $\wh{\alpha}(X\scr_{\alpha}G)\sub (X\scr_{\alpha}G)\ft L(G)  $.
   	
\end{proof}

The next result proves that, for any $ \LI $-comodule $ X $, both the Fubini crossed product and the spatial crossed product are independent of the Hilbert space on which $ X $ is represented.  	

\begin{pro}[Uniqueness of the crossed product]\label{pro2.14} Let $ (X,\alpha) $ and $ (Y,\beta) $ be two $ \LI $-comodules and suppose that $ X $ and $ Y $ are w*-closed subspaces of $ B(H) $ and $ B(K) $ respectively. If there exists an $ \LI $-comodule isomorphism $ \Phi\colon X\to Y $, then the isomorphism $\Psi:= \Phi\ot \mathrm{id }_{\BLT}\colon X\vt\BLT\to Y\vt\BLT $ is an $ \LI $-comodule isomorphism from $ (X\vt\BLT,\wt{\alpha}) $ onto $ (Y\vt\BLT,\wt{\beta}) $, which maps $ X\fcr_{\alpha}G $ onto $ Y\fcr_{\beta}G $ and $ X\scr_{\alpha}G $ onto $ Y\scr_{\beta}G $. Also, $ \Psi|_{X\fcr_{\alpha}G} $ is an $ L(G) $-comodule isomorphism from $ (X\fcr_{\alpha}G,\wh{\alpha}) $ onto  $ (Y\fcr_{\beta}G,\wh{\beta})$ and $ \Psi|_{X\scr_{\alpha}G} $ is an $ L(G) $-comodule isomorphism from $ (X\scr_{\alpha}G,\wh{\alpha}) $ onto  $ (Y\scr_{\beta}G,\wh{\beta})$. Furthermore, $ \Psi $ is an $ L(G) $-bimodule map, i.e. $ \Psi((1_{H}\ot \lambda_{s})x(1_{H}\ot\lambda_{t}))=(1_{K}\ot \lambda_{s})\Psi(x)(1_{K}\ot\lambda_{t}) $, for all $ s,t\in G $ and $ x\in X\vt\BLT $.
\end{pro}
\begin{proof} First, since $ \Phi $ is a comodule morphism we have that $ \beta\circ\Phi=(\Phi\ot id)\circ\alpha $ and hence:
	\begin{align*} 
	   \wt{\beta}\circ\Psi&=(id\ot \mathrm{Ad }U_{G}^{*})\circ(id\ot\sigma)\circ(\beta\ot id)\circ(\Phi\ot id)\\ 
	   &=(id\ot \mathrm{Ad }U_{G}^{*})\circ(id\ot\sigma)\circ((\beta\circ\Phi)\ot id)\\
	   &=(id\ot \mathrm{Ad }U_{G}^{*})\circ(id\ot\sigma)\circ\left[ ((\Phi\ot id)\circ\alpha)\ot id\right] \\
	   &=(id\ot \mathrm{Ad }U_{G}^{*})\circ(id\ot\sigma)\circ(\Phi\ot id\ot id)\circ(\alpha\ot id)\\
	   &=(\Phi\ot id\ot id)\circ(id\ot \mathrm{Ad }U_{G}^{*})\circ(id\ot\sigma)\circ(\alpha\ot id)\\
	   &=(\Psi\ot id)\circ\wt{\alpha},
	\end{align*}
which proves that $ \Psi $ is an $ \LI $-comodule isomorphism from $ (X\vt\BLT,\wt{\alpha}) $ onto $ (Y\vt\BLT,\wt{\beta}) $. This implies that $ \Psi $ maps the fixed point subspace $ X\fcr_{\alpha}G $ of $ \wt{\alpha} $ onto the fixed point subspace $ Y\fcr_{\beta}G $ of $ \wt{\beta} $. On the other hand, the relation $ \beta\circ\Phi=(\Phi\ot id)\circ\alpha $ yields that
\[\Psi(\alpha(X))=(\Phi\ot id)(\alpha(X))=\beta(\Phi(X))=\beta(Y) \]
and since $ \Psi $ is an $ L(G) $-bimodule map (see Remark \ref{rem5}) it follows that $ \Psi $ maps $ X\scr_{\alpha}G $ onto $ Y\scr_{\beta}G $. It remains to show that
\[\wh{\beta}\circ\Psi=(\Psi\ot id)\circ\wh{\alpha}. \]
Indeed:
   \begin{align*} 
      \wh{\beta}\circ\Psi=(\mathrm{id }_{Y}\ot\delta_{G})\circ(\Phi\ot \mathrm{id }_{\BLT})&=(\Phi\ot \mathrm{id }_{\BLT}\ot \mathrm{id }_{L(G)})\circ(\mathrm{id }_{X}\ot\delta_{G})\\
      &=(\Psi\ot \mathrm{id }_{L(G)})\circ\wh{\alpha}.
   \end{align*}	
\end{proof}

The equality of the spatial and Fubini crossed products does not pass to sub-comodules. 
Indeed, any   $\LI $-comodule is isomorphic to a sub-comodule of a von Neumann algebra (see Remark \ref{rem1}), 
and the latter will always satisfy the equality condition.
However, if a sub-comodule is covariantly complemented, 
i.e. if it is the range of a projection commuting with the action, the situation is better. 
In a more general setting, the following proposition gives a sufficient condition so that $ X\fcr_{\alpha}G=X\scr_{\alpha}G $, for an $ \LI $-comodule $ (X,\alpha) $.

\begin{pro}\label{pro3.9} Let $ (X,\alpha) $ and $ (Y,\beta) $ be two $ \LI $-comodules and suppose that $ X $ and $ Y $ are w*-closed subspaces of $ B(H) $ and $ B(K) $ respectively.
	Suppose that $\zeta\colon X\to Y$  is an $ \LI $-comodule monomorphism and there exists an $ \LI $-comodule morphism $ \phi\colon Y\to X $ onto $X$ such that $\phi\circ\zeta=\mathrm{id}_X$. 
	If  $ Y\fcr_{\beta}G=Y\scr_{\beta}G $, then $ X\fcr_{\alpha}G=X\scr_{\alpha}G $.
\end{pro}
\begin{proof} Let $ \psi:=\phi\ot \mathrm{id }_{\BLT}\colon Y\vt\BLT\to X\vt\BLT $. Then, with the same argument as in the proof of Proposition \ref{pro2.14}, we get that $ \psi $ is an $ \LI $-comodule morphism from $ (Y\vt\BLT,\wt{\beta}) $ to $ (X\vt\BLT,\wt{\alpha}) $ and since $ \phi $ is onto $ X $, it follows that $ \psi $ is onto $ X\vt\BLT $. Similarly, the map $ \theta:=\zeta\ot \mathrm{id }_{\BLT}\colon X\vt\BLT\to Y\vt\BLT $ is an $ \LI $-comodule monomorphism such that $ \psi\circ\theta=\mathrm{id }_{X\vt\BLT} $.
	
	\noindent We will show that $ \psi $ maps $ Y\fcr_{\beta}G $ onto $ X\fcr_{\alpha}G $. Indeed, on the one hand,  since $ \psi $ is an $ \LI $-comodule morphism from  $ (Y\vt\BLT,\wt{\beta}) $ to $ (X\vt\BLT,\wt{\alpha}) $ it maps the fixed point subspace of $ (Y\vt\BLT,\wt{\beta}) $ into the fixed point subspace of $ (X\vt\BLT,\wt{\alpha}) $, that is $ \psi(Y\fcr_{\beta}G)\sub X\fcr_{\alpha}G $. On the other hand, for any $ x\in X\fcr_{\alpha}G $, we have $ x=\psi(y) $, where $ y:=\theta(x)\in Y\vt\BLT $. Thus we get: 
	\begin{align*} 
	   \wt{\beta}(y)&=\wt{\beta}(\theta(x))=(\theta\ot\mathrm{id }_{\LI})(\wt{\alpha}(x))\\
	   &=(\theta\ot\mathrm{id }_{\LI})(x\ot 1)\\&=\theta(x)\ot 1=y\ot1
	   \end{align*}
	and thus $ y\in Y\fcr_{\beta}G $. This yields that $ X\fcr_{\alpha}G\sub\psi(Y\fcr_{\beta}G) $ and therefore we have the equality $ X\fcr_{\alpha}G=\psi(Y\fcr_{\beta}G) $.
	
	\noindent Now if we assume that $ Y\fcr_{\beta}G=Y\scr_{\beta}G $, then we have:
	\begin{align*} 
	X\fcr_{\alpha}G&=\psi\left(Y\fcr_{\beta}G \right)\\
	&=\psi\left(Y\scr_{\beta}G \right)\\
	&\sub\wsp\{\psi\left( (1_{K}\ot\lambda_{s})\beta(y)\right):\ s\in G,\ y\in Y \}\\
	&=\wsp\{(\phi\ot id)\left( (1_{K}\ot\lambda_{s})\beta(y)\right):\ s\in G,\ y\in Y \}\\
	&=\wsp\{(1_{H}\ot\lambda_{s})(\phi\ot id)\left(\beta(y)\right):\ s\in G,\ y\in Y \}\\
	&=\wsp\{(1_{H}\ot\lambda_{s})\alpha(\phi(y)):\ s\in G,\ y\in Y \}\\
	&=\wsp\{(1_{H}\ot\lambda_{s})\alpha(x):\ s\in G,\ x\in X \}\\
	&=X\scr_{\alpha}G 
	\end{align*} 	and therefore $ X\fcr_{\alpha}G=X\scr_{\alpha}G $, because the inclusion $ X\scr_{\alpha}G\sub X\fcr_{\alpha}G $ holds in general (see Proposition \ref{pro2.13}).
	
\end{proof}

\begin{remark}\label{rem6} Let $ Y $ and $ Z $ be two W*-TRO's. A W*-TRO-morphism $ \beta\colon Y\to Z $ is called \bfc{non-degenerate} if the linear spans of $ \beta(Y)Z^{*}Z $ and $ \beta(Y)^{*}ZZ^{*} $ are w*-dense respectively in $ Z $ and $ Z^{*} $. 
	
	\noindent In \cite{SS}, P. Salmi and A. Skalski proved essentially that $ X\fcr_{\alpha}G=X\scr_{\alpha}G $ when $ X $ is a W*-TRO and the action $ \alpha $ is in addition a non-degenerate W*-TRO morphism. One alternative way to prove this result is the following. Consider the linking von Neumann algebra $ R_{X} $ and the canonical w*-continuous projection $ P\colon R_{X}\to X $. Then, the action $ \alpha $ extends uniquely to a W*-$ \LI $-action $ \beta\colon R_{X}\to R_{X}\vt\LI $ since $ \alpha $ is a non-degenerate TRO-morphism (see \cite[Proposition 1.2 and Theorem 2.3]{SS}), such that $ \alpha\circ P=(P\ot id)\circ\beta $. Therefore, $ X\fcr_{\alpha}G=X\scr_{\alpha}G $ follows from Proposition \ref{pro3.9}.
	
\end{remark}

\begin{pro}\label{pro2.15} For any $ \LI $-comodule $ (X,\alpha) $, it holds that \[(X\fcr_{\alpha}G)^{\wh{\alpha}}=(X\scr_{\alpha}G)^{\wh{\alpha}}=\alpha(X). \]	
\end{pro}
\begin{proof} By Proposition \ref{pro2.14} and Remark \ref{rem1}, we may assume that $ (X,\alpha) $ is a subcomodule of a W*-$ \LI $-comodule $ (M,\alpha) $, i.e. $ M $ is a von Neumann algebra that contains $ X $ and the action $ \alpha $ extends to a W*-$ \LI $-action on $ M $, which we denote again by $ \alpha $. Then, since $ \wt{\alpha} $ commutes with $ \mathrm{id }_{M}\ot\delta_{G} $ (see Proposition \ref{pro4.2}), we get:
	\begin{align*}  
	   (X\fcr_{\alpha}G)^{\wh{\alpha}}&=\left( (X\vt\BLT)^{\wt{\alpha}}\right)^{\wh{\alpha}}\\
	   &=\left( (X\vt\BLT)^{\mathrm{id }_{X}\ot\delta_{G}}\right)^{\wt{\alpha}} \\
	   &=\left( X\vt(\BLT)^{\delta_{G}} \right)^{\wt{\alpha}}\\
	   &=(X\vt\LI)^{\wt{\alpha}} ,
	\end{align*}
because $ \BLT^{\delta_{G}}=\LI $ (see e.g. \cite[\S 0.3.9]{SVZ1}).	By Theorem \ref{thm1.1}, there exists a w*-group action $ \gamma\colon G\to \mathrm{Aut}(M) $, such that $ \pi_{\gamma}=\alpha $. Therefore, by Proposition \ref{pro4.1} and Lemma \ref{lem3.2} we get that $ (M\vt\LI)^{\wt{\alpha}}=\alpha(M) $ and so we have:
    \begin{align*} 
       (X\vt\LI)^{\wt{\alpha}}&=(M\vt\LI)^{\wt{\alpha}}\cap(X\vt \LI)\\
       &=\alpha(M)\cap(X\vt \LI).
    \end{align*}
Thus, it suffices to show that $ \alpha(M)\cap(X\vt \LI)=\alpha(X) $. The inclusion $ \alpha(M)\cap(X\vt \LI)\supseteq\alpha(X) $ is obvious. To prove the inclusion $ \alpha(M)\cap(X\vt \LI)\sub\alpha(X) $ consider an $ x\in M $, such that $ \alpha(x)\in X\vt\LI $. Then $ (\mathrm{id }_{M}\ot f)(\alpha(x))\in X $ for all $ f\in \LO $ and hence, for any $ \omega\in M_{*} $ with $ \omega|_{X}=0 $, we have:
    \begin{align*}
       &\la (\mathrm{id }_{M}\ot f)(\alpha(x)),\ \omega\ra   =0,\quad\forall f\in \LO\\
       \implies&\la\alpha(x),\ \omega\ot f \ra=0,\quad\forall f\in \LO\\
       \implies&\int_{G}f(s)\la\gamma_{s^{-1}}(x),\ \omega\ra\ ds=0, \quad\forall f\in \LO,
    \end{align*} 	
    because $ \alpha(x)(s) =\pi_{\gamma}(x)(s)=\gamma_{s^{-1}}(x)$, for all $ s\in G $. Since the function $ G\ni s\mapsto \la\gamma_{s^{-1}}(x),\ \omega\ra$ is bounded and continuous, the last equality implies that it is identically zero. Therefore, by taking $ s=e $ (where $ e $ is the unit of $ G $), we get that $ \la x,\ \omega\ra=0 $, for any $ \omega\in M_{*} $ with $ \omega|_{X}=0 $. Thus, $ x\in X $. So, we have proved that $ (X\fcr_{\alpha}G)^{\wh{\alpha}}=\alpha(X) $.\\
On the other hand, we have $ (X\scr_{\alpha}G)^{\wh{\alpha}}=(X\fcr_{\alpha}G)^{\wh{\alpha}}\cap(X\scr_{\alpha}G)=\alpha(X)\cap(X\scr_{\alpha}G)=\alpha(X) $ and the proof is complete.
 	
\end{proof} 

\begin{defin}\label{def2.16} Let $ (M,\Delta) $ be a Hopf von Neumann algebra acting on a Hilbert space $ K $ and $ (X,\alpha) $ be an $ M $-comodule with $ X $ being a w*-closed subspace of some $ B(H) $. We say that $ (X,\alpha) $ is \bfc{non-degenerate} if
	\[X\vt B(K)=\wsp\{(1_{H}\ot y)\alpha(x):\ y\in B(K),\ x\in X \}. \]	
\end{defin} 

\begin{lem}\label{lem2.17} Let $ (X,\alpha) $ be an $ M $-comodule for a Hopf von Neumann algebra $ (M,\Delta) $. If $ (X,\alpha) $ is non-degenerate, then 
	\[X=\wsp\{\omega\cdot x:\ \omega\in M_{*},\ x\in X \}, \]
where $ \omega\cdot x=(\mathrm{id }_{X}\ot\omega)(\alpha(x)) $ is the corresponding $ M_{*}$-module action on $ X $ (see Remark \ref{rem3})	
\end{lem}
\begin{proof} Suppose that $ (X,\alpha) $ is non-degenerate and that the von Neumann algebra $ M $ acts on a Hilbert space $ K $ and $ X$ is w*-closed in $ B(H) $ for some Hilbert space $ H $. Let $ \phi\in X_{*} $, such that $ \phi(\omega\cdot x)=0 $, for all $ \omega\in M_{*}$ and $ x\in X $. Then, we have:
	\begin{align*} 
	   &\phi\circ(\mathrm{id }_{X}\ot\omega)\circ\alpha(x)=0,\quad \forall\omega\in M_{*},\ x\in X\\
	   \implies&\omega\circ(\phi\ot \mathrm{id }_{B(K)})\circ\alpha(x)=0,\quad \forall\omega\in M_{*},\ x\in X\\
	   \implies&(\phi\ot \mathrm{id }_{B(K)})\circ\alpha(x)=0,\quad \forall x\in X\\
	   \implies& b(\phi\ot \mathrm{id }_{B(K)})\circ\alpha(x)=0,\quad \forall b\in B(K),\ x\in X\\
	   \implies&(\phi\ot \mathrm{id }_{B(K)})\left((1_{H}\ot b) \alpha(x)\right) =0,\quad \forall b\in B(K),\ x\in X.\\
	\end{align*}
Since $ (X,\alpha) $ is non-degenerate, the last condition implies that $ (\phi\ot \mathrm{id }_{B(K)})(y)=0 $ for any $ y\in X\vt B(K) $; thus $ \phi(x)1=(\phi\ot \mathrm{id }_{B(K)})(x\ot1)=0 $	for any $ x\in X $ and hence $ \phi=0 $. So the desired conclusion follows from the Hahn-Banach theorem.

\end{proof}  

For the proof of the next result see \cite[Lemma II.1.4 and Corollary II.1.5]{SVZ2}.


\begin{pro}\label{cor6.1}\cite[Corollary II.1.5]{SVZ2} Let $ \delta\colon N\to N\vt L(G) $ be a W*-$ L(G) $-action on a von Neumann algebra $ N $. For any $ x\in N $ and any $ k\in A(G) $, we have
	\[(k\cdot x)\ot1_{\LT}\in\wsp\{(1_{N}\ot\lambda_{s})\delta(h\cdot k\cdot x):\ s\in G,\ h\in A(G) \}, \]
where $ k\cdot x=(\mathrm{id }_{N}\ot k)(\delta(x)) $, for $ k\in A(G) $ and $ x\in N $.	
\end{pro}

\begin{cor}\label{cor6.2} Let $ \delta\colon N\to N\vt L(G) $ be a W*-$ L(G) $-action on a von Neumann algebra $ N\sub B(K) $ and let $ Y\sub N $ be an $ L(G) $-subcomodule of $ N $ (i.e. $ Y $ is w*-closed subspace of $ N $ and $ \delta(Y)\sub Y\ft L(G) $). Then, the following are equivalent:
	\begin{itemize}
		\item[(i)] $ Y=\wsp\{h\cdot y:\ h\in A(G),\ y\in Y \} $,
		\item[(ii)] $ (Y,\delta) $ is non-degenerate,
		\end{itemize}
where $ h\cdot y=(\mathrm{id }_{Y}\ot h)(\delta(y)) $, for $ h\in A(G) $ and $ y\in Y $.		
\end{cor}
\begin{proof} The implication (ii)$ \implies $(i) is immediate by Lemma \ref{lem2.17}. 
	
	(i)$ \implies $(ii): Since $ Y=\wsp\{h\cdot y:\ h\in A(G),\ y\in Y \} $, from Proposition \ref{cor6.1} it follows that
	\[z\ot1_{\LT}\in\wsp\{(1_{K}\otimes b)\delta(y):\ b\in \BLT,\ y\in Y \}, \] 
	for any $ z\in Y $.\\
	Therefore, for any $ z\in Y $ and $ c\in \BLT $, we have that
	\[ z\ot c=(1_{K}\ot c)(z\ot 1_{\LT})\in\wsp\{(1_{K}\otimes b)\delta(y):\ b\in \BLT,\ y\in Y \}, \] 
	because the multiplication in $ B(K)\vt\BLT $ is separately w*-continuous.\\ Thus, we have that 
	\[Y\vt\BLT\sub\wsp\{(1_{K}\otimes b)\delta(y):\ b\in \BLT,\ y\in Y \}. \]
	The converse inclusion follows from $ \delta(Y)\sub Y\ft L(G)\sub Y\vt\BLT $ and the fact that $ Y\vt\BLT $ is a $ \CIK\vt\BLT $-bimodule (see Remark \ref{rem5}).
	
\end{proof}

In what follows, for a Hilbert space $ H $ and any subsets $ A,\ B\sub B(H) $, we will denote by $ AB $ the set of all products $ ab $, where $ a\in A $ and $ b\in B $.

The next result gives a characterization of those $ \LI $-comodules for which  the corresponding Fubini crossed product is equal to the spatial crossed product.

\begin{thm}\label{mainthm} Let $ (X,\alpha) $ be an $ \LI $-comodule. The following are equivalent:
	\begin{itemize} 
		\item[(i)] $ X\fcr_{\alpha}G=X\scr_{\alpha}G $,
		\item[(ii)] $ (X\fcr_{\alpha}G,\wh{\alpha}) $ is a non-degenerate $ L(G) $-comodule,
		\item[(iii)] $ X\fcr_{\alpha}G=\wsp\{h\cdot y:\ h\in A(G),\ y\in X\fcr_{\alpha}G \}, $
	\end{itemize} 
where 	$ h\cdot y=(\mathrm{id }_{X\fcr_{\alpha}G}\ot h)(\wh{\alpha}(y)) $, for  $ h\in A(G),\ y\in X\fcr_{\alpha}G $.	
\end{thm}
\begin{proof} By Proposition \ref{pro2.14} and Remark \ref{rem1}, we may assume that $ (X,\alpha) $ is a subcomodule of a W*-$ \LI $-comodule $ (M,\alpha) $, with $ M $ acting on a Hilbert space $ H $.  Also, let us fix some notation. Set $ K:=H\otimes\LT $, $ N:=B(K) $, $ Y:=X\fcr_{\alpha}G $, $ Y_{0}:=X\scr_{\alpha}G$ and $ \delta:=\mathrm{id }_{B(H)}\otimes\delta_{G}\colon N\to N\vt L(G) $.
	
The equivalence (ii)$ \iff $(iii) is immediate from Corollary \ref{cor6.2}.	
	
	\medskip\noindent
(ii)$ \implies $(i): Suppose that $ (X\fcr_{\alpha}G,\wh{\alpha}) $ is non-degenerate and consider the following faithful *-representation of $ L(G) $ on $ K $:
\[u\colon L(G)\to B(K):\ u(x)=1_{H}\otimes x. \]
Then, observe that 
\begin{itemize}
	\item[(1)] $ u(L(G))Yu(L(G))\sub Y $, i.e. $ Y $ is a $ \CI\vt L(G) $-subbimodule of $ N $,
	\item[(2)] $ \delta\circ u=(u\otimes \mathrm{id }_{L(G)})\circ\delta_{G} $, i.e. $ u $ is a comodule monomorphism from $ (L(G),\delta_{G}) $ to $ (N,\delta) $.
\end{itemize}	
The representation $ u $ defines a unitary operator $ R\colon K\otimes\LT\to K\otimes\LT $ by the formula
\[(R\xi)(s)=u(\lambda_{s^{-1}})(\xi(s))=(1_{H}\otimes\lambda_{s^{-1}})(\xi(s)), \]
for $ s\in G$ and $ \xi\in L^{2}(G,K)\simeq K\otimes\LT $.\\
We claim that $ R\in\CI\vt L(G)\vt\LI $. Indeed, if we take $ T\in B(H) $, $ r\in G $ and $ f\in\LI $, then, for every $ \eta\in H,\ \phi,\ \psi\in\LT $ and $ s\in G $, we have
\begin{align*} 
(R(T\ot\rho_{r}\ot f)(\eta\ot\phi\ot\psi))(s)&=(R(T\eta\ot\rho_{r}\phi\ot f\psi))(s)\\
&=(1_{H}\ot\lambda_{s^{-1}})(f(s)\psi(s)(T\eta\ot\rho_{r}\phi))\\
&=f(s)\psi(s)(T\eta\ot\lambda_{s^{-1}}\rho_{r}\phi)\\
&=f(s)\psi(s)(T\eta\ot\rho_{r}\lambda_{s^{-1}}\phi)\\
&=(T\ot\rho_{r}\ot f)(\eta\ot\lambda_{s^{-1}}\phi\ot\psi)(s)\\
&=((T\ot\rho_{r}\ot f)R)(\eta\ot\phi\ot\psi))(s)
\end{align*}
thus $ R\in (B(H)\vt R(G)\vt\LI)'=\CI\vt L(G)\vt\LI $.\\
Combining that $ R\in\CI\vt L(G)\vt\LI $ with (1) and Remark \ref{rem5}, yields that the normal *-isomorphism \[ \mathrm{Ad }R\colon N\vt\BLT\to N\vt\BLT:\ T\mapsto RTR^{*} \] maps $ Y\vt\BLT $ onto $ Y\vt\BLT $.\\
Now, we consider two W*-$ L(G) $-actions on the von Neumann algebra $ N\vt\BLT $, namely
\[\wt{\delta}:=(\mathrm{id }_{N}\otimes \mathrm{Ad }W_{G})\circ(\mathrm{id }_{N}\otimes\sigma)\circ(\delta\otimes \mathrm{id }_{\BLT}) \]
and \[\bar{\delta}:=(\mathrm{id }_{N}\otimes\sigma)\circ(\delta\otimes \mathrm{id }_{\BLT}). \] 
Let $ \Sigma\in B(K\otimes L^{2}(G\times G)) $ with $ \Sigma\xi(s,t)=\xi(t,s) $, for $ \xi\in K\otimes L^{2}(G\times G) $. Then, for any $ \xi\in K\otimes L^{2}(G\times G) $, we have:
\begin{align*} 
\bar{\delta}(R)\xi(s,t)&=\Sigma(\delta\otimes \mathrm{id }_{\BLT})(R)(\Sigma\xi)(s,t)=(\delta\otimes \mathrm{id }_{\BLT})(R)(\Sigma\xi)(t,s)\\
&=\delta(u(\lambda_{s^{-1}}))(\Sigma\xi)(t,s)=(u\otimes \mathrm{id }_{L(G)})(\delta_{G}(\lambda_{s^{-1}}))(\Sigma\xi)(t,s)\\
&=(u(\lambda_{s^{-1}})\otimes\lambda_{s^{-1}})(\Sigma\xi)(t,s)=u(\lambda_{s^{-1}})(\Sigma\xi)(st,s)=u(\lambda_{s^{-1}})\xi(s,st)\\
&=(R\otimes1)(1\otimes W_{G})\xi(s,t)
\end{align*}
Thus, $ \bar{\delta}(R)=(R\otimes1)(1\otimes W_{G}) $, which implies that $ \mathrm{Ad }R $ is an $ L(G) $-comodule isomorphism from $ (N\vt\BLT,\wt{\delta})$ onto $ (N\vt\BLT,\bar{\delta})$. Indeed, for any $ T\in N\vt\BLT $, we have
\begin{align*} 
(\bar{\delta}\circ \mathrm{Ad }R)(T)&=\bar{\delta}(R)\bar{\delta}(T)\bar{\delta}(R)^{*}=(R\otimes1)(1\otimes W_{G})\bar{\delta}(T)(1\otimes W_{G}^{*})(R^{*}\otimes 1)\\
&=(\mathrm{Ad }R\otimes \mathrm{id }_{\BLT})\circ(\mathrm{id }_{N}\otimes \mathrm{Ad }W_{G})\circ\bar{\delta}(T)\\
&=(\mathrm{Ad }R\otimes \mathrm{id }_{\BLT})\circ\wt{\delta}(T)
\end{align*}
Therefore, $ \mathrm{Ad }R $ preserves fixed points, i.e.
\[\mathrm{Ad }R\left( (N\vt\BLT)^{\wt{\delta}} \right)=\left(N\vt\BLT \right ) ^{\bar{\delta}}. \]
On the other hand, $N^{\delta}\vt\BLT= \left(N\vt\BLT \right ) ^{\bar{\delta}} $. Indeed, for every $ x\in N^{\delta} $ and $ b\in\BLT $, we have
\[\bar{\delta}(x\otimes b)=(\mathrm{id }_{N}\otimes\sigma)(\delta(x)\otimes b)=(\mathrm{id }_{N}\otimes\sigma)(x\otimes1\otimes b)=x\otimes b\otimes1, \]
hence we have the inclusion $ N^{\delta}\vt\BLT\sub \left(N\vt\BLT \right ) ^{\bar{\delta}}  $. For the converse inclusion, take $ T\in N\vt\BLT $ such that $ \bar{\delta}(T)=T\otimes1 $. Then, for $ \omega\in\BLT_{*} $, $ \phi\in N_{*} $ and $ h\in A(G) $, we have
\begin{align*} 
\langle \delta\circ(\mathrm{id }_{N}\otimes\omega)(T),\ \phi\otimes h \rangle&=\langle (\mathrm{id }_{N}\otimes\omega)(T),\ (\phi\otimes h)\circ\delta \rangle\\
&=\langle T,\ (\phi\otimes h\otimes\omega)\circ(\delta\otimes \mathrm{id }_{\BLT}) \rangle\\
&=\langle T,\ (\phi\otimes \omega\otimes h)\circ(\mathrm{id }_{N}\otimes\sigma)\circ(\delta\otimes \mathrm{id }_{\BLT}) \rangle\\
&=\langle \bar{\delta}(T),\ \phi\otimes \omega\otimes h \rangle=\langle T\otimes1,\ \phi\otimes \omega\otimes h \rangle\\
&=\langle T,\ \phi\otimes \omega\rangle \langle 1,\ h \rangle=\langle (id\otimes\omega)(T)\otimes1,\ \phi\otimes h \rangle
\end{align*}    
thus $ (\mathrm{id }_{N}\otimes\omega)(T)\in N^{\delta} $ for all $ \omega\in\BLT_{*} $, which implies that $ T\in N^{\delta}\vt\BLT $.\\
Furthermore, by Proposition \ref{pro2.15}, we have $ \alpha(X)=Y^{\wh{\alpha}} $. Also, recall that $ \BLT=\wsp\{yf:\ f\in\LI,\ y\in L(G) \} $ and observe that $ (\CIK\vt\LI)\delta(N)\sub(N\vt\BLT)^{\wt{\delta}} $. Indeed, if $ y\in N $ and $ f\in\LI $, then
\begin{align*} 
\wt{\delta}(\delta(y))&=(\mathrm{id }_{N}\otimes \mathrm{Ad }W_{G})\circ(\mathrm{id }_{N}\otimes\sigma)\circ(\delta\otimes \mathrm{id }_{\BLT})(\delta(y))\\
&=(\mathrm{id }_{N}\otimes \mathrm{Ad }W_{G})\circ(\mathrm{id }_{N}\otimes\sigma)\circ(\mathrm{id }_{N}\otimes\delta_{G})(\delta(y))\\
&=(\mathrm{id }_{N}\otimes \mathrm{Ad }W_{G})\circ(\mathrm{id }_{N}\otimes\delta_{G})(\delta(y))\qquad(\text{because }\sigma\circ\delta_{G}=\delta_{G})\\
&=(\mathrm{id }_{N}\otimes \mathrm{Ad }W_{G})\circ(\mathrm{id }_{N}\otimes \mathrm{Ad }W_{G}^{*})(\delta(y)\otimes1)\quad(\delta_{G}(x)=W_{G}^{*}(x\otimes1)W_{G})\\
&=\delta(y)\otimes1 
\end{align*}
and
\begin{align*} 
\wt{\delta}(1_{K}\otimes f)&=(\mathrm{id }_{N}\otimes \mathrm{Ad }W_{G})\circ(\mathrm{id }_{N}\otimes\sigma)\circ(\delta\otimes \mathrm{id }_{\BLT})(1_{K}\otimes f)\\
&=(\mathrm{id }_{N}\otimes \mathrm{Ad }W_{G})(1_{K}\otimes f\otimes1_{\LT})\\&=1_{K}\otimes W_{G}^{*}(f\otimes1_{\LT})W_{G}\\
&=1_{K}\otimes f\otimes1_{\LT} \qquad\quad (W_{G}\in\LI\vt L(G))
\end{align*}
By assumption $ (Y,\delta) $ is non-degenerate, i.e.
\[Y\vt\BLT=\wsp\{(1_{K}\otimes b)\delta(x):\ x\in Y,\ b\in\BLT \}\]
and $ \BLT=\wsp\left\{ L(G)\LI  \right\} $. Thus it follows that
\begin{align*}
Y\vt\BLT&=\wsp\{(1_{K}\otimes y)(1_{K}\otimes f)\delta(x):\ x\in Y,\ y\in L(G),\ f\in\LI \}\\
&\sub\wsp\left\lbrace (1_{K}\otimes y)T:\ y\in L(G),\ T\in (Y\vt\BLT)^{\wt{\delta}} \right\rbrace. \tag{*}
\end{align*}
The last inclusion follows from $ (\CIK\vt\LI)\delta(N)\sub(N\vt\BLT)^{\wt{\delta}} $ and the fact that $ (Y\vt\BLT,\wt{\delta}) $ is an $ L(G) $-subcomodule of $ (N\vt\BLT,\wt{\delta}) $.\\
Therefore, since $ \mathrm{Ad }R $ maps $ Y\vt\BLT $ onto $ Y\vt\BLT $, the above inclusion (*)  implies that
\begin{align*}  
Y\vt\BLT&=R(Y\vt\BLT)R^{*}\\
&\sub\wsp\left\lbrace R(\CIK\vt L(G))R^{*}R(Y\vt\BLT)^{\wt{\delta}}R^{*} \right\rbrace.
\end{align*}
On the other hand, we have
\begin{align*} 
R(Y\vt\BLT)^{\wt{\delta}}R^{*}&=\left(R(N\vt\BLT)^{\wt{\delta}}R^{*}\right)\cap(Y\vt\BLT)\\
&=(N\vt\BLT)^{\bar{\delta}}\cap(Y\vt\BLT)\\
&=(N^{\delta}\vt\BLT)\cap(Y\vt\BLT)\\&=((N^{\delta}\cap Y)\vt\BLT)\\
&=Y^{\delta}\vt\BLT\\
&=Y^{\wh{\alpha}}\vt\BLT=\alpha(X)\vt\BLT
\end{align*}
and \[ R(\CIK\vt L(G))R^{*}\sub\CI\vt L(G)\vt\BLT, \] because $ R\in\CI\vt L(G)\vt\BLT $.\\
Hence, $ Y\vt\BLT\sub\wsp\left\lbrace R(\CIK\vt L(G))R^{*}R(Y\vt\BLT)^{\wt{\delta}}R^{*} \right\rbrace $ yields that    
\begin{align*}     
Y\vt\BLT&\sub\wsp\left\lbrace \left(\CI\vt L(G)\vt\BLT\right)\left(\alpha(X)\vt\BLT\right) \right\rbrace\\
&\sub\left(\wsp\left\lbrace \left(\CI\vt L(G)\right)\alpha(X) \right\rbrace\right)\vt\BLT\\
&=Y_{0}\vt\BLT.
\end{align*} 
It follows that $ Y=Y_{0} $, that is $ X\fcr_{\alpha}G=X\scr_{\alpha}G $.
\medskip

\noindent
(i)$ \implies $(ii): Suppose that condition (i) holds and keep the same notation as above. By Remark \ref{rem5} we have that $ Y\vt\BLT $ is a $ \CIK\vt\BLT $-bimodule. Thus, since $ \wh{\alpha}(Y)\sub Y\ft L(G)\sub Y\vt\BLT $, we have the inclusion $ Y\vt\BLT\supseteq\wsp\{(1_{K}\ot b)\wh{\alpha}(y):\ b\in\BLT,\ y\in Y \} $.

For the converse inclusion, it suffices to show that $ Y_{0}\vt\BLT\sub \wsp\{(1_{K}\ot b)\wh{\alpha}(y):\ b\in\BLT,\ y\in Y_{0} \} $, since $ Y=Y_{0} $ (by the assumption that (i) holds). Indeed, for any $ s\in G $, $ x\in X $ and $b\in\BLT$, we have
\begin{align*} 
((1_{H}\ot\lambda_{s})\alpha(x))\ot b&=(1_{H}\ot1_{\LT}\ot b\lambda_{s}^{-1})(((1_{H}\ot\lambda_{s})\alpha(x))\ot\lambda_{s})\\
&=(1_{H}\ot1_{\LT}\ot b\lambda_{s}^{-1})\wh{\alpha}((1_{H}\ot\lambda_{s})\alpha(x)).
\end{align*}
Thus, $ (1_{H}\ot\lambda_{s})\alpha(x)\ot b\in\wsp\{(1_{K}\ot c) \wh{\alpha}(y):\ c\in\BLT,\ y\in Y_{0} \} $. Since $ Y_{0}\vt\BLT $ is the w*-closed linear span of the elements of the form $ (1_{H}\ot\lambda_{s})\alpha(x)\ot b $, we get the desired inclusion.

\end{proof}

\begin{remark}\label{rem10} Note that from the last part of the proof of Theorem \ref{mainthm} (the proof of the implication (i)$ \implies $(ii)) it follows that the spatial crossed product $ (X\scr_{\alpha}G,\wh{\alpha}) $ is always a non-degenerate $ L(G) $-comodule.	
\end{remark}

\begin{remark}\label{rem8} The result of  Salmi and Skalski (see Remark \ref{rem6}) can also be proved using Theorem \ref{mainthm} as follows.
	
	\noindent Let $ X $ be a W*-TRO, $ \alpha\colon X\to X\vt \LI $ be an $ \LI $-action on $ X $ such that $ \alpha $ is a non-degenerate W*-TRO-morphism (see Remark \ref{rem6} for the definition) and consider the canonical w*-continuous projection $ P\colon R_{X}\to X $, where $ R_{X} $ is the linking von Neumann algebra of $ X $. Also, let $ \beta\colon R_{X}\to R_{X}\vt\LI $ be the unique W*-$ \LI $-action that extends $ \alpha $. Then, by the proof of Proposition \ref{pro2.14} it follows that $ P\ot \mathrm{id }_{\BLT}\colon R_{X}\vt$ $ \BLT\to X\vt \BLT $ is also a w*-continuous projection and an $ \LI $-comodule morphism with respect to the actions $ \wt{\beta} $ and $ \wt{\alpha} $. Therefore, as in the proof of Proposition \ref{pro3.9} it follows that $ P\ot \mathrm{id }_{\BLT} $ restricts to a w*-projection $ Q\colon R_{X}\fcr_{\beta}G\to X\fcr_{\alpha}G $. Furthermore, following the last part of the proof of Proposition \ref{pro2.14} we get the equality \[ \wh{\alpha}\circ Q=(Q\ot \mathrm{id }_{L(G)})\circ\wh{\beta}, \]  which by Remark \ref{rem3} means exactly that  \[ Q(u\cdot y)=u\cdot Q(y)  \text{ for any }  u\in A(G)  \text{ and }  y\in R_{X}\fcr_{\beta}G. \]  Thus $ Q\colon R_{X}\fcr_{\beta}G\to X\fcr_{\alpha}G  $ is a w*-continuous $ A(G) $-module morphism onto $ X\fcr_{\alpha}G $.  On the other hand, we have $ R_{X}\fcr_{\beta}G=\wsp\{A(G)\cdot (R_{X}\fcr_{\beta}G) \} $ because of the Digernes-Takesaki theorem and Theorem \ref{mainthm}. It follows that $ X\fcr_{\alpha}G=\wsp\{A(G)\cdot (X\fcr_{\alpha}G) \} $ and hence $ X\fcr_{\alpha}G=X\scr_{\alpha}G $ by Theorem \ref{mainthm}.
\end{remark}
  
\section{An application to groups with the approximation property}\label{sec4}

Let $ G $ be a locally compact group. A complex-valued function $ u\colon G\to \mathbb{C} $ is called a \bfc{multiplier} for the Fourier algebra $ A(G) $ if the linear map $  m_{u}(v)=uv $
maps $ A(G) $ into $ A(G) $. The space of all multipliers of $ A(G) $ is denoted by $ MA(G) $ and $ MA(G)\sub C_{b}(G) $. For $ u\in MA(G) $, we denote by $ M_{u} $ the w*-continuous linear map from $ L(G) $ to $ L(G) $ defined by $ M_{u}=m_{u}^{*} $. The function $ u $ is called a \bfc{completely bounded multiplier} if $ M_{u} $ is completely bounded. The space of all completely bounded multipliers is denoted by $ M_{0}A(G) $ and it is a Banach space with the norm $ ||u||_{M_{0}}=||M_{u}||_{cb} $. Moreover, $ A(G)\sub M_{0}A(G) $.

It is known that $ M_{0}A(G) $ is the dual Banach space of $ Q(G) $, where $Q(G)$ denotes the completion of $ \LO $ in the norm
\[||f||_{Q}=\sup\left\{\left|\int_{G}f(s)u(s)\ ds \right|: u\in M_{0}A(G),\ ||u||_{M_{0}}\leq1 \right\}. \] 

We say that $ G $ has \bfc{the approximation property} (AP) if there is a net $ \{u_{i}\}_{i\in I} $ in $ A(G) $, such that $ u_{i}\rightarrow1 $ in the $ \sigma(M_{0}A(G),Q(G)) $-topology.

For more details on completely bounded multipliers and the aproximation property see for example  \cite{dCH}, \cite{CH} and \cite{HK}.

We will need the following theorem due to U. Haagerup and J. Kraus (see \cite[Proposition 1.7 and Theorem 1.9]{HK}).

\begin{thm}\label{thm7.1} For any locally compact group $ G $, the following conditions are equivalent:
	\begin{itemize}
		\item[(i)] $ G $ has the AP.
		\item[(ii)] There is a net $ \{u_{i}\} $ in $ A(G) $, such that, for any von Neumann algebra $ N $, $ (\mathrm{id }_{N}\ot M_{u_{i}})(x)\longrightarrow x$ in the w*-topology for all $ x\in N\vt L(G) $.
	\end{itemize}	
\end{thm}

\begin{pro}\label{pro7.1} Let $ X\sub B(H) $ be a w*-closed subspace of $ B(H) $, $ G $ be a locally compact group with the AP and $ \alpha\colon X\to X\vt L^{\ap}(G) $ be an $ \LI $-action. Then, for any $ x\in X\fcr_{\alpha}G $, we have:
	\[x\in \overline{A(G)\cdot x}^{\text{w*}}, \]
	where $ h\cdot x=(\mathrm{id }_{X\fcr_{\alpha}G}\ot h)\circ\wh{\alpha}(x) $, for any $ x\in X\fcr_{\alpha}G $ and $ h\in A(G) $. 	
\end{pro}
\begin{proof} Let us denote $ Y:=X\fcr_{\alpha}G $ and $ K:=H\ot\LT $. First, we prove that the dual action $ \wh{\alpha}\colon Y\to Y\ft L(G) $ satisfies the condition
	\begin{equation} 
	(\mathrm{id }_{Y}\ot M_{u})\circ\wh{\alpha}=\wh{\alpha}\circ(\mathrm{id }_{Y}\ot u)\circ\wh{\alpha},
	\end{equation}
	for any $ u\in A(G) $.\\
	Indeed, observe that for any $ u,\ h\in A(G) $ and $ y\in L(G) $ we have:
	\begin{align*} 
	\la M_{u}(y), h\ra&=\la y, hu\ra\\
	&=\la\delta_{G}(y),h\ot u\ra\\
	&=\la(\mathrm{id }_{L(G)}\ot u)\circ\delta_{G}(y),h\ra,      
	\end{align*}
	therefore $ M_{u}=(\mathrm{id }_{L(G)}\ot u)\circ\delta_{G} $, for any $ u\in A(G) $. So, for any  $ u\in A(G) $, we have:
	\begin{align*} \wh{\alpha}\circ(\mathrm{id }_{Y}\ot u)\circ\wh{\alpha}&=(\mathrm{id }_{Y}\ot \mathrm{id }_{L(G)}\ot u)\circ(\wh{\alpha}\ot \mathrm{id }_{L(G)})\circ\wh{\alpha}\\
	&=(\mathrm{id }_{Y}\ot \mathrm{id }_{L(G)}\ot u)\circ(\mathrm{id }_{Y}\ot\delta_{G}) \circ\wh{\alpha}\\
	&=\left[ \mathrm{id }_{Y}\ot\left(  (\mathrm{id }_{L(G)}\ot u)\circ\delta_{G}\right)  \right]\circ\wh{\alpha}\\
	&= (\mathrm{id }_{Y}\ot M_{u})\circ\wh{\alpha},
	\end{align*}
	and (1) is proved.\\	
  Since $ G $ has the AP, by Theorem \ref{thm7.1} there exists a net $ \{u_{i} \} $ in $ A(G)$ such that $ (\mathrm{id }_{B(K)}\ot M_{u_{i}})(y)\longrightarrow y$ ultraweakly for all $ y\in B(K)\vt L(G) $.\\
   Therefore, $(\mathrm{id }_{Y}\ot M_{u_{i}})(\wh{\alpha}(x))\longrightarrow \wh{\alpha}(x)$ ultraweakly, for any $ x\in Y $, because $ \wh{\alpha}(Y)\sub Y\ft L(G)\sub B(K)\vt L(G) $. Thus (1) implies that $ \wh{\alpha}\circ(\mathrm{id }_{Y}\ot u_{i})\circ\wh{\alpha}(x)\longrightarrow \wh{\alpha}(x)$ ultraweakly, for any $ x\in Y $. On the other hand, $ \wh{\alpha} $ is a w*-continuous isometry, therefore it is a w*-w*-homeomorphism from $ Y $ onto $ \wh{\alpha}(Y) $ (see e.g. \cite{BLM}, Theorem A.2.5) and thus $(\mathrm{id }_{Y}\ot u_{i})\circ\wh{\alpha}(x)\longrightarrow x$ ultraweakly, that is $ u_{i}\cdot x\longrightarrow x $ ultraweakly, for any $ x\in Y $. 
	
\end{proof}

The next result is due to J. Crann and M. Neufang (see \cite{CN}). Here, we give an alternative proof as an application of Theorem \ref{mainthm} and Proposition \ref{pro7.1}.

\begin{pro}\label{pro7.2}  Let  $ G $ be a locally compact group with the AP and $ (X,\alpha) $  an $ \LI $-comodule. Then, \[ X\fcr _{\alpha}G=X\scr_{\alpha}G. \]	
\end{pro} 
\begin{proof} Since $ G $ has the AP, by Proposition \ref{pro7.1} it follows that $ y\in \overline{A(G)\cdot y}^{\text{w*}}$ for any $ y\in X\fcr_{\alpha}G $, which implies that:
	\[X\fcr_{\alpha}G=\wsp\{h\cdot y:\ y\in X\fcr_{\alpha}G,\ h\in A(G) \}. \] 
Therefore, Theorem \ref{mainthm} yields that $ X\fcr _{\alpha}G=X\scr_{\alpha}G $. 
	
\end{proof} 
    
\section{Crossed products and ideals of $ \LO $}     
  
For this section, let $ J $ be a closed left ideal of $ \LO $ and consider its annihilator $ J^{\perp}\sub\LI $.  In \cite{AKT}, M. Anoussis, A. Katavolos and I. Todorov define two $ L(G) $-subbimodules of $ \BLT $, namely $ \B $ and $ \R $ and ask whether these bimodules are equal. They proved that this is the case when $ G $ is either weakly amenable discrete or compact or abelian. 

This result was later generalized by J. Crann and M. Neufang \cite{CN} who proved that if $ G $ has the AP, then $ \B=\R $. Their approach is based on their non-commutative Fej\'{e}r-type theorem for crossed products of von Neumann algebras by groups with the AP (see \cite{CN}).

Here, we prove that $ \B $ and $ \R $ can be realized respectively as the spatial crossed product and the Fubini crossed product of a certain $ \LI $-action on $ J^{\perp} $. Thus, Theorem \ref{mainthm} provides a necessary and sufficient condition for the equality $ \B=\R $.

For any $ h\in\LO $, the map $ \Theta(h)\colon\BLT\to\BLT $ is defined by
\[\Theta(h)(T)=\int_{G} h(s)\mathrm{Ad }\rho_{s}(T)\ ds, \]
where the integral is understood in the w*-topology. We define 
\[\R=\ker\Theta(J):=\{T\in\BLT:\ \Theta(h)(T)=0,\ \forall h\in J \}. \]
On the other hand, $ \B $ is the normal $ L(G) $-bimodule generated by $ J^{\perp} $, that is
\[\B=\wsp\{\lambda_{s}f\lambda_{t}:\ s,\ t\in G,\ f\in J^{\perp} \}. \]
Then, $ \B $ and $ \R $ are $ L(G) $-bimodules and $ \B \sub \R $. For more details regarding $ \B $ and $ \R $, as well as for theoriginal definition of $ \R $, see \cite{AKT}. 

Consider the normal *-monomorphism $ \Phi\colon\BLT\to\BLT\vt\BLT $ given by
\[\Phi(T)=V_{G}^{*}\delta_{G}(T)V_{G}=V_{G}^{*}W^{*}_{G}(T\ot1)W_{G}V_{G}. \]
Then, it is not hard to see that \begin{equation}\label{e1}
\Phi(f)=\alpha_{G}(f),\quad f\in\LI
\end{equation} and \begin{equation}\label{e2}
\Phi(\lambda_{s})=1\ot\lambda_{s},\quad s\in G.
\end{equation}
Also, we have
\[\wh{\alpha_{G}}\circ\Phi=(\Phi\ot \mathrm{id }_{L(G)})\circ\delta_{G} \]  
and therefore, $ \Phi $ is a W*-$ L(G) $-comodule isomorphism from $ (\BLT ,\delta_{G})$ onto $ (\LI\scr_{\alpha_{G}}G,\wh{\alpha_{G}}) $.

Since $ J $ is a closed left ideal of $ \LO $, it follows that $ \alpha_{G}(J^{\perp})\sub J^{\perp}\vt\LI $, i.e. $ J^{\perp} $ is an $ \LI $-subcomodule of $ (\LI,\alpha_{G}) $. 

\begin{pro}\label{pro1} The $ L(G) $-comodule isomorphism $ \Phi $ defined above maps $ \B $ onto $ J^{\perp}\scr_{\alpha_{G}}G $ and $ \R $ onto  $ J^{\perp}\fcr_{\alpha_{G}}G $.	In particular, $ \B $ and $ \R $ are $ L(G) $-subcomodules of $ (\BLT,\delta_{G}) $.
\end{pro}
\begin{proof} First, note that from the covariance relations 
	\[\lambda_{s}f\lambda_{s^{-1}}=f_{s},\quad s\in G, \]
	(where $ f_{s}(t)=f(s^{-1}t) $)  we get that 
	\[\B=\wsp\{\lambda_{s}f:\ s\in G,\ f\in J^{\perp} \}, \]
	thus \eqref{e1} and \eqref{e2} imply that $ \Phi(\B)=J^{\perp}\scr_{\alpha_{G}}G $.\\
	It remains to show that $ \Phi(\R)=J^{\perp}\fcr_{\alpha_{G}}G $.\\
	Note that
	  \begin{align*} 
	      J^{\perp}\fcr_{\alpha_{G}}G&=(J^{\perp}\vt\BLT)^{\wt{\alpha_{G}}}\\
	      &=(\LI\vt\BLT)^{\wt{\alpha_{G}}}\cap(J^{\perp}\vt\BLT)\\
	      &=(\LI\scr_{\alpha_{G}}G)\cap(J^{\perp}\vt\BLT),
	  \end{align*}
	since $ (\LI\vt\BLT)^{\wt{\alpha_{G}}}=(\LI\scr_{\alpha_{G}}G) $ by the Digernes-Takesaki theorem.\\	
	Therefore if $ y\in \LI\scr_{\alpha_{G}}G $, then
	\[y\in J^{\perp}\fcr_{\alpha_{G}}G\iff(h\ot \mathrm{id }_{\BLT})(y)=0,\ \forall h\in J. \]
	Since $ \R $ is the intersection of the kernels of the maps $ \Theta(h) $ for $ h\in J $, and  $ J^{\perp}\fcr_{\alpha_{G}}G $ is the intersection of the kernels of the maps $ (h\ot \mathrm{id }_{\BLT}) $, for $ h\in J $, it suffices to prove that
	\[\Theta(h)=(h\ot \mathrm{id }_{\BLT})\circ\Phi,\quad\forall h\in \LO. \] 
	Since both maps are linear and w*-continuous, it suffices to prove the equality for elements of the form $ f\lambda_{s} $, for $ f\in\LI $ and $ s\in G $, whose linear span is w*-dense in $ \BLT $. But $ \Theta(h) $ and $ (h\ot \mathrm{id }_{\BLT})\circ\Phi $ are $ L(G) $-module maps since clearly \[\Theta(h)(f\lambda_{s})=\Theta(h)(f)\lambda_{s}\] and
	\[(h\ot id)(\Phi(f\lambda_{s}))=(h\ot id)(\Phi(f)(1\ot\lambda_{s}))=(h\ot id)(\Phi(f))\lambda_{s},\]
	therefore it suffices to prove that \[\Theta(h)(f)=(h\ot id)(\Phi(f))\text{ , for all }  h\in \LO  \text{ and }  f\in\LI.\]
	Indeed, first observe that $ \Theta(h)(f)=f_{h} $, where $ f_{h}(t)=\int_{G}h(s)f(ts)ds $. Thus, for any $ k\in \LO $,
	\begin{align*} \la \Theta(h)(f),\ k \ra&=\la f_{h},\ k \ra\\&=\int_{G}f_{h}(t)k(t)dt\\&=\iint_{G\times G}h(s)f(ts)k(t)dsdt\\&=\la\alpha_{G}(f),\ h\ot k\ra\\
	&=\la(h\ot id)(\alpha_{G}(f)),\  k\ra\\
	&=\la(h\ot id)(\Phi(f)),\  k\ra
	\end{align*}
	and so, $ \Theta(h)(f)=(h\ot id)(\Phi(f)) $ and the proof is complete.

\end{proof}

\begin{remark}\label{rem9} Following \cite{NRS}, let $ CB^{\sigma,L(G)}_{\LI}(\BLT) $ be the algebra of all w*-continuous completely bounded $ \LI $-bimodule morphisms on $ \BLT $ which leave $ L(G) $ invariant. Also, let $ V_{\mathrm{inv}}^{b}(G) $ and $ V_{\mathrm{inv}}^{\ap}(G) $ be the space of continuous right invariant Schur multipliers and the space of measurable right invariant Schur multipliers respectively (see \cite[section 4]{NRS} for the precise definitions).  It is shown in \cite[Theorem 4.3]{NRS} that there are  completely isometric isomorphisms 
	\[ M_{0}A(G)\cong V_{\mathrm{inv}}^{b}(G)\cong V_{\mathrm{inv}}^{\ap}(G)\cong CB^{\sigma,L(G)}_{\LI}(\BLT) \]
In particular, every completely bounded multiplier $ u\in M_{0}A(G) $ is identified with a map $ S_{u}\in CB^{\sigma,L(G)}_{\LI}(\BLT) $ (sometimes called a Herz-Schur multiplier), which is the unique element in $ CB^{\sigma,L(G)}_{\LI}(\BLT) $ that extends the completely bounded and w*-continuous map $ M_{u}=m_{u}^{*}\colon L(G)\to L(G) $ defined in section \ref{sec4}.

\noindent In the case of an element $ u\in A(G) $, the map $ S_{u} $ can be described in terms of the $ L(G) $-action $ \delta_{G} $ on $ \BLT $. Namely, we have:
\[S_{u}=(\mathrm{id }_{\BLT}\ot u)\circ\delta_{G}:\BLT\to\BLT,\]
for any $ u\in A(G) $.

\noindent Indeed, since $ \delta_{G}(f)=f\ot 1 $ and $ \delta_{G}(\lambda_{s})=\lambda_{s}\ot\lambda_{s} $ for any $ f\in \LI $ and $ s\in G $, it follows easily that 	 
	\[(\mathrm{id }_{\BLT}\ot u)\circ\delta_{G}(f\lambda_{s})=u(s)f\lambda_{s},\qquad s\in G,\ f\in\LI,\ u\in A(G). \]
On the other hand, $ M_{u}(\lambda_{s})=u(s)\lambda_{s} $ for any $ s\in G $ and therefore if $ u\in A(G) $, then $ (\mathrm{id }_{\BLT}\ot u)\circ\delta_{G} $ is the unique completely bounded w*-continuous $ \LI $-bimodule morphism on $ \BLT $ that extends $ M_{u} $. That is $ S_{u}=(\mathrm{id }_{\BLT}\ot u)\circ\delta_{G} $ for all $ u\in A(G) $.

\noindent Therefore, the $ A(G) $-module structure of  $\BLT $ induced by the $ L(G) $-action $ \delta_{G}\colon \BLT\to \BLT\vt L(G) $ (recall Remark \ref{rem3}) can be described in terms of Schur multipliers $ S_{u} $ with $ u\in A(G) $, that is
\[S_{u}(T)=(\mathrm{id }_{\BLT}\ot u)\circ\delta_{G}(T)=u\cdot T, \]
for all $ u\in A(G)  $ and $ T\in\BLT $.
\end{remark}

The next corollary, which is an immediate consequence of Theorem \ref{mainthm} and Proposition \ref{pro1}, provides a necessary and sufficient condition so that $ \B$ $ = \R $. 

\begin{cor}\label{cor1} The following conditions are equivalent:
	\begin{itemize}
		\item[(a)] $ \B=\R $,
		\item[(b)] $ (\R,\delta_{G}) $ is a non-degenerate $ L(G) $-comodule, 
		\item[(c)] $ \R=\wsp\{S_{u}(T):\ u\in A(G),\ T\in\R \} $.
	\end{itemize}	
\end{cor}
\begin{proof} The equivalence between (a) and (b) is clear from Theorem \ref{mainthm} and Proposition \ref{pro1}. Also, the equivalence between (b) and (c) follows again from Theorem \ref{mainthm}, because $ S_{u}(T)=(\mathrm{id }_{\BLT}\ot u)\circ\delta_{G}(T)=u\cdot T $, for any $ u\in A(G)  $ and $ T\in\BLT $.
	
\end{proof}

\begin{remark} Note that from Proposition \ref{pro1} and Proposition \ref{pro7.1} it follows that if $ G $ has the AP, then there exists a net $ (u_{i})_{i\in I} $ in $ A(G) $, such that
	\[S_{u_{i}}(T)\longrightarrow T,\text{ ultraweakly for all }T\in\BLT. \]
Therefore, Corollary \ref{cor1} implies that if $ G $ has the AP, then $ \B=\R $.
\end{remark}

\section*{Acknowledgements}
The author would like to thank Aristides Katavolos for his constant encouragement and many fruitful discussions and  Michael Anoussis for helpful suggestions and discussions. The research work was supported by the Hellenic Foundation for Research and Innovation (HFRI) and the General Secretariat for Research and Technology (GSRT), under the HFRI PhD Fellowship Grant (GA. no. 74159/2017).


\end{document}